\documentclass[letterpaper, 10pt, conference]{cssconf}

\IEEEoverridecommandlockouts

\usepackage{times}
\usepackage{amsmath}
\usepackage{amsfonts}       
\usepackage{graphicx}
\usepackage{hyperref}

\usepackage{acro}

\usepackage{tikz}
\usepackage{pgfplots}
 \usetikzlibrary{
    angles,
 	arrows,
 	arrows.meta,
 	shapes,
 	shapes.symbols,
 	shapes.callouts,
 	shapes.geometric,
 	arrows,
 	matrix,
 	backgrounds,
 	positioning,
 	plotmarks,
 	calc,
 	patterns,
 	matrix,
 	decorations.pathreplacing,
 	decorations.pathmorphing,
 	decorations.text,
 	decorations.shapes,
 	decorations.fractals,
 	decorations.markings,
 	spy,
    quotes
 }
\usepgfplotslibrary{fillbetween}
\usepgfplotslibrary{statistics}
\usepgfplotslibrary{groupplots}

\pgfplotsset{
  /pgfplots/confidence box/.style 2 args={
    legend image code/.code={
        \definecolor{steelblue31119180}{RGB}{31,119,180}
        \draw[steelblue31119180,no markers, fill=steelblue31119180, opacity=0.5]
        plot coordinates {
        (-0.1cm,-0.1cm)
        (-0.1cm,0.2cm)
        (0.5cm,0.2cm)
        (0.5cm,-0.1cm)
        (-0.1cm,-0.1cm)
      }
      node[rectangle]{};
    }
  }
}

\usepackage{tikz}
\usepackage{graphicx}
\usetikzlibrary{shapes,arrows}
\usepackage{pgfplots}
\pgfplotsset{compat=newest}
\usepackage{amsmath, amsfonts, bbold, amssymb}
\usepackage[noend]{algpseudocode}
\usepackage{algorithm2e}
\acsetup{patch/maketitle = false}

\newtheorem{theorem}{Theorem}
\definecolor{darkgray176}{RGB}{176,176,176}
\definecolor{darkorange25512714}{RGB}{255,127,14}
\definecolor{forestgreen4416044}{RGB}{44,160,44}
\definecolor{steelblue31119180}{RGB}{31,119,180}

\usepackage{dsfont}

\DeclareAcronym{aic}{
    short = AIC,
    long = Akaike Information Criterion 
}

\DeclareAcronym{aicc}{
    short = AICc,
    long = Akaike Information Criterion correction 
}

\DeclareAcronym{bfgs}{
    short = BFGS,
    long = Broyden-Fletcher-Goldfarb-Shanno algorithm
}

\DeclareAcronym{bic}{
    short = BIC,
    long = Bayesian Information Criterion 
}

\DeclareAcronym{cks}{
    short = CKS,
    long = Compositional Kernel Search 
}

\DeclareAcronym{dl}{
    short = DL,
    long = Deep Learning 
}

\DeclareAcronym{gp}{
    short = GP,
    long = Gaussian Process,
    long-plural-form = Gaussian Processes 
}

\DeclareAcronym{hmc}{
    short = HMC, 
    long = Hamiltonian Monte Carlo
}

\DeclareAcronym{kl}{
    short = KL, 
    long = Kullback-Leibler
}

\DeclareAcronym{ks}{
    short = KS,
    long = Kernel Search 
}

\DeclareAcronym{lfm}{
    short = LFM, 
    long = Latent Force Model 
}

\DeclareAcronym{lodegp}{
    short = LODE-GP, 
    long = Linear Ordinary Differential Equation Gaussian Process,
    long-plural-form = Linear Ordinary Differential Equation Gaussian Processes
}

\DeclareAcronym{lti}{
    short = LTI, 
    long = Linear Time Invariant 
}

\DeclareAcronym{map}{
    short = MAP, 
    long = Maximum A Posteriori 
}

\DeclareAcronym{mc}{
    short = MC, 
    long = Monte Carlo 
}

\DeclareAcronym{mcmc}{
    short = MCMC, 
    long = Markov Chain Monte Carlo 
}

\DeclareAcronym{mll}{
    short = MLL, 
    long = Marginal Log Likelihood 
}

\DeclareAcronym{ml}{
    short = ML, 
    long = Machine Learning
}

\DeclareAcronym{mpc}{
    short = MPC, 
    long = Model Predictive Control 
}

\DeclareAcronym{nn}{
    short = NN, 
    long = Neural Network 
}

\DeclareAcronym{nuts}{
    short = NUTS, 
    long = No U-Turn Sampler 
}

\DeclareAcronym{ode}{
    short = ODE, 
    long = Ordinary Differential Equation 
}

\DeclareAcronym{rmse}{
    short = rmse, 
    long = Root Mean Squared Error 
}

\DeclareAcronym{se}{
    short = SE, 
    long = Squared Exponential 
}

\DeclareAcronym{sgd}{
    short = SGD, 
    long = Stochastic Gradient Descent 
}

\DeclareAcronym{skc}{
    short = SKC, 
    long = Structured Kernel Composition 
}

\DeclareAcronym{snf}{
    short = SNF, 
    long = Smith Normal Form 
}

\DeclareAcronym{rkhs}{
    short = RKHS, 
    long = Reproducing Kernel Hilbert Space 
}

\title{Physics-informed Gaussian Processes as Linear Model Predictive Controller with Constraint Satisfaction}
\author{Jörn Tebbe$^{1}$, Andreas Besginow$^{1}$, Markus Lange-Hegermann$^{1}$
\thanks{$^{1}$All authors are with Institute Industrial IT, Faculty of Electrical Engineering and Automation, OWL University of Applied Sciences and Arts, Lemgo, Germany.
Jörn Tebbe and Andreas Besginow are supported by the SAIL project which is funded by the Ministry of Culture and Science of the State of North Rhine-Westphalia under the grant no NW21-059C. Corresponding author: joern.tebbe@th-owl.de}}

\begin{document}

\maketitle

\begin{abstract}
    Model Predictive Control evolved as the state of the art paradigm for safety critical control tasks.
    Control-as-Inference approaches thereof model the constrained optimization problem as a probabilistic inference problem.
    The constraints have to be implemented into the inference model. 
    A recently introduced physics-informed Gaussian Process method uses Control-as-Inference with a Gaussian likelihood for state constraint modeling, but lacks guarantees of open-loop constraint satisfaction. 
    We mitigate the lack of guarantees via an additional sampling step using Hamiltonian Monte Carlo sampling in order to obtain safe rollouts of the open-loop dynamics which are then used to obtain an approximation of the truncated normal distribution which has full probability mass in the safe area.
    We provide formal guarantees of constraint satisfaction while maintaining the ODE structure of the Gaussian Process on a discretized grid. Moreover, we show that we are able to perform optimization of a quadratic cost function by closed form Gaussian Process computations only and introduce the Matérn kernel into the inference model.
\end{abstract}

\section{Introduction}

Controlling technical systems is a central part of modern industrial applications.
The physical limits of real world systems have to be considered in order to provide safe operation \cite{tebbe2024efficiently}.
Model Predictive Control \cite{rawlings2017model} is a compelling framework for this, which requires a model of the system which can be derived from first order principles and/or data-driven approaches.
This model, e.g.\ a \ac{gp} \cite{rasmussen2006gaussian}, projects the system under given inputs into the future and we obtain the optimal control input via a constrained optimization scheme.
The optimization is constrained on the dynamics, the initial state and input, and on input and state constraints for the considered horizon. \\
Another approach of MPC is the Control-as-Inference \cite{levine2018reinforcement} paradigm which suggests obtaining the desired control input for the next time step by solving a probabilistic inference problem instead of an optimization.
In recent work, such a scheme was introduced using \acp{gp} \cite{tebbe2025physics}.
As a Bayesian method, a \ac{gp} prior uses an inductive bias on the modeled functions and provides closed predictive distribution formulas under the assumption of Gaussian noise. \\
For the Control-as-Inference scheme, the authors used a \ac{gp} prior for \acp{ode}, which model the underlying control system.
The control input is obtained via conditioning the \ac{gp} on a desired reference. The obtained marginal predictive distributions provide each point in the output space with strict positive probability mass which makes this framework unsuitable for regression bounded by given constraints. Although the authors choose the mean of the predictive distribution as the control input, it is not guaranteed that the mean satisfies the given constraints. \\
A better choice to model bounded functions, which is required for MPC, is pushing full probability mass inside the constraints. This procedure is well known in the \ac{gp} literature for ordinary regression \cite{jensen2013bounded}.
The easiest approach is warping \cite{snelson2003warped}, due to its preservation of the Gaussian likelihood. 
Unfortunately, the nonlinear transformation breaks the covariance structure of the linear \ac{ode} kernel which means that both mean and samples no longer have to fulfill the \ac{ode}.
Therefore, we proceed to use bounded likelihood formulations which break Gaussianity, therefore require approximate solutions like \ac{mc} approximations, but retain the \ac{lodegp} prior unchanged. \\
The \ac{mc} methods use samples from the posterior and apply accept/reject sampling for the trajectories in favor of the constraints. This naive approach lacks efficiency, since the probability mass of the constrained space will degrade with growing dimension of the \ac{gp}. We therefore chose an advanced sampling strategy such as \ac{hmc} \cite{pakman2014exact} which is able to approximate small subsets of the output space. \\
As the authors in \cite{tebbe2025physics} only use \ac{se} kernel functions as part of their latent \ac{gp} prior for the parameterizing functions, the sampling for constraint satisfaction has to break the thereby induced smoothness. As a remedy, we propose to use Matérn type kernel functions with suitable smoothness parameter in order to obtain more flexible functions which are able to fulfill the constraints.
Moreover, we want to optimize cost functions, which are often of quadratic type. These cost functions provide easy to compute control laws \cite{rawlings2017model} for linear \ac{mpc} problems. In a similar manner, we provide a way of incorporating quadratic cost functions in a Control-as-Inference fashion for \acp{gp}. \\
The remainder of the paper is structured as follows. In Section \ref{sec:preliminaries} we provide our considered \ac{mpc} problem and recall Control-as-Inference via \acp{gp}. In Section \ref{sec:method} we provide our methodology for the choice of a better kernel, an additional reweighting as adaptation of the cost function and a subsequent sampling step for open-loop constraint satisfaction. In Section \ref{sec:experiments} we provide two numerical examples which validate our method, while we conclude in Section \ref{sec:conclusion}.
\section{Problem formulation and Preliminaries}
\label{sec:preliminaries}

\subsection{Problem formulation}

We want to solve the optimal control problem
\begin{subequations} \label{eq:mpc}
\begin{align} 
    \min_{u(t)} &\sum_{i=0}^{N} (x_{\text{ref}}(t_i) - x(t_i))^2 + \Vert u \Vert\label{eq:mpc:obj}\\
    \text{s.t. } \partial_t x &= Ax + Bu, \label{eq:mpc:con:ode}\\
    x(t_0) &= x_0, \label{eq:mpc:con:init}\\
    x_{\min} &\leq x(t) \leq x_{\max} \quad \forall t \in [t_0,t_N], \label{eq:mpc:con:x}\\
    u_{\min} &\leq u(t) \leq u_{\max} \quad \forall t \in [t_0,t_N] \label{eq:mpc:con:u}.
\end{align}
\end{subequations}
with $A \in \mathbb{R}^{n_x \times n_x}$ and $B \in \mathbb{R}^{n_x \times n_u}
$ as system and control matrices, respectively \cite{rawlings2017model}. The objective \eqref{eq:mpc:obj} is a quadratic cost function for a given reference with regularization of the control function. The linear \ac{ode} \eqref{eq:mpc:con:ode} and the initial state \eqref{eq:mpc:con:init} are assumed to be known in advance. For ease of notation, we assume that the state bounds \eqref{eq:mpc:con:x} and input bounds \eqref{eq:mpc:con:u} are constant. The proposed methods also work for time-varying bounds which we validate in section \ref{sec:experiments}. We assume that the control system $(A,B)$ is controllable. \\
For \ac{mpc}, this problem is solved repeatedly between two time steps $t_i$ and $t_{i+1}$, while the computed control action is applied for $t_{i+1}$ to $t_{i+2}$. 

\subsection{Gaussian Processes}
A \acf{gp} \cite{rasmussen2006gaussian} $g = \mathcal{GP}(\mu, k)$ is a stochastic process with the property that all function evaluations $g(t_1), \ldots, g(t_n)$ are jointly Gaussian.
Such a \ac{gp} is fully characterized by its mean function $\mu(t)$ and covariance function $k(t, t')$.
We can condition the \ac{gp} prior on data $\mathcal{D}=(T,Z)$ ,with $T \in \mathbb{R}^{n \times 1}$ and $Z\in \mathbb{R}^{n \times n_z}$, with a Gaussian likelihood and obtain the posterior via Bayes' rule
\begin{equation}
    p(f|T,Z,\theta) = \frac{p(Z|f,T,\theta)p(f, T, \theta)}{p(T,Z,\theta)}
\end{equation}
with hyperparameters $\theta \in \mathbb{R}^{n_\theta}$.
Since both prior and likelihood are Gaussian, the product is also Gaussian. We thus obtain again a \ac{gp} as posterior distribution which is again fully defined by its mean and covariance which are derived as
\begin{subequations} \label{eq:gp}
    \begin{align}
        \mu^* &= \mu(t^*) + K_*^\top(K + \sigma_n^2 I)^{-1}(z - \mu(t))\label{eq:gp:mean} \\
        \Sigma^* &= K_{**} - K_*^\top(K + \sigma_n^2 I)^{-1} {K_*} \label{eq:gp:cov}
    \end{align}
\end{subequations}
with covariance matrices $K = (k(t_i, t_j))_{i,j} \in \mathbb{R}^{n \times n}$, $K_* = (k(t_i, t^*_j))_{i, j} \in \mathbb{R}^{n \times n_*}$ and $K_{**} = (k(t^*_i, t^*_j))_{i, j} \in \mathbb{R}^{n_* \times n_*}$ for predictive positions $t^* \in \mathbb{R}^{n_*}$ with noise variance $\sigma_n^2$.
The noise variance does not have to be constant (homoscedastic), but can be input dependent (heteroscedastic), i.e. $\sigma_n^2(t) \in \mathbb{R}^{n_z}$. The noise variance describes the noise we expect on a given datapoint $(t_i,z_i)$.
In our case we use the time $t$ as input and concatenate control state and input as one variable, i.e. $(x,u) = z \in \mathbb{R}^{n_z}$ with $n_z = n_x+n_u$, for the output of the \ac{gp}. Therefore, we obtain a predictive distribution of size $n_z \cdot n_*$. \\  
Hyperparameters are commonly introduced via the covariance function, for example the \ac{se} 
covariance function often includes signal variance $\sigma_f^2$ and smoothness parameter $\ell^2$: 
\begin{equation}\label{eq:SE_kernel}
    k_{\text{SE}}(t, t') = \sigma_f^2\exp\left(-\frac{(t-t')^2}{2\ell^2}\right)
\end{equation}
As the function $k_\text{SE}$ is infinitely many times differentiable, so are the samples of the \ac{gp} almost surely. 
Another class of covariance functions is the Matérn family
\begin{align*}
    k_\nu(t,t') &= \sigma_f^2 \exp\left(-\frac{\sqrt{2r + 1}\vert t - t'\vert}{\ell}\right) v(x) \\
    v(x) &= \frac{r!}{(2r)!} \sum_{i=0}^r \frac{(r+i)!}{i!(r-i)!}\left(\frac{2\sqrt{2r+1}\vert t - t'\vert}{\ell} \right)^{r-i}
\end{align*}
    for $\nu = r + \frac{1}{2}$ and $\ell > 0$. The \ac{gp} prior $\mathcal{GP}(0, k_{r+\frac{1}{2}})$ is $r$ times mean squared differentiable \cite{rasmussen2006gaussian}, in contrast to its samples which are not necessarily $r$ times differentiable.
    Note, that the \ac{se} kernel is obtained as a Matérn kernel for $\nu \to \infty$ and is therefore infinitely many times differentiable.
    
\subsection{Linear Ordinary Differential Equation GPs}\label{sec:LODE_GP}
The class of \acp{gp} is closed under linear operations,
 i.e.\ applying a linear operator $\mathcal{L}$ to a \ac{gp} $g$ as $\mathcal{L}g$ is again a \ac{gp}.
This ensures that realizations of the \ac{gp} $\mathcal{L}g$ lie in the image of the linear operator $\mathcal{L}$ \cite{langehegermann2018algorithmic}.
We demonstrate the procedure for constructing so-called \acp{lodegp} --- \acp{gp} that strictly satisfy the underlying system of linear homogeneous ordinary differential equations \eqref{eq:mpc:con:ode}.
We subtract $(\partial_t x)I$ and combine state $x$ and input $u$ by stacking it in one variable $z \in \mathbb{R}^{n_z}$ to reformulate the system in \eqref{eq:mpc:con:ode} as:
\begin{align}\label{eq:harmonic_oscillator_diffeq}
	0 = H \cdot z =  \begin{pmatrix}
	    A - \partial_t x I | B
	\end{pmatrix} \cdot 
    \begin{pmatrix}
        x \\
        u
    \end{pmatrix}
\end{align}
We can algorithmically factor $H \in \mathbb{R}[\partial_t]^{n_x \times n_z}$ into three matrices such that $Q\cdot H\cdot V = D$, where $D \in \mathbb{R}[\partial_t]^{n_x \times n_z}$ is called Smith Normal Form and $Q \in \mathbb{R}[\partial_t]^{n_x \times n_x}, V \in \mathbb{R}[\partial_t]^{n_z \times n_z}$ invertible \cite{newman1997smith}.
All matrices are defined over the polynomial ring $\mathbb{R}[\partial_t]$ i.e.\ their entries are polynomials in $\partial_t$.

We then construct a prior latent \ac{gp} $\tilde{g} = \mathcal{GP}(0, \tilde{K})$ with $\tilde{K} \in \mathbb{R}^{n_z \times n_z}$ being a matrix of latent kernels, using the simple construction rules presented in Table 1 of \cite{besginow2022constraining}. The construction is based on the diagonal entries of $D$ which are $0$ or $1$ for controllable systems.
We obtain a \ac{lodegp} by applying the linear operator $V$ to this \emph{latent} \ac{gp} $\tilde{g}$ via
\begin{equation}\label{eq:LODE_GP_harmonic}
    V\tilde{g} = \mathcal{GP}\left(\mathbf{0}, V\cdot 
			\tilde{K}
    \cdot \hat{V}^\top\right)
\end{equation}
where $\hat{V}$ is the operator $V$ applied to the second argument ($t'$) of the covariance functions $k_i$ in $\tilde{K}$.
We thereby guarantee that the realizations of the resulting \ac{lodegp} $V\tilde{g}$ in Equation~\eqref{eq:LODE_GP_harmonic} \emph{strictly} satisfy the system in Equation~\eqref{eq:harmonic_oscillator_diffeq}, as detailed in \cite{besginow2022constraining}.
The \ac{lodegp} in Equation~\eqref{eq:LODE_GP_harmonic} can be trained and conditioned on datapoints as any regular \ac{gp}.

\subsection{Control as Inference with \acp{lodegp}} \label{sec:CaI}

Recall the optimal control problem \eqref{eq:mpc}. This problem is solved in \cite{tebbe2025physics} via a Control-as-Inference scheme using a \ac{lodegp}.
The minimization of \eqref{eq:mpc:obj} is obtained via the correspondence of optimality in \acp{rkhs} and the GP posterior mean \eqref{eq:gp:mean} \cite{tebbe2025physics}.
The ODE constraint \eqref{eq:mpc:con:ode} is fulfilled exactly via the \ac{lodegp} structure.
For the initial condition constraint \eqref{eq:mpc:con:init}, we use the initial point $(t_0, z_0)=(t_0, (x_0,u_0))$ as datapoint in the conditioned dataset $\mathcal{D}$ with a heteroscedastic noise of the numerical jitter $\sigma_n^2(t_0)=10^{-8}$.
The two constraints \eqref{eq:mpc:con:ode} and \eqref{eq:mpc:con:init} require no additional tuning. \\
The critical constraints are the state and input constraints \eqref{eq:mpc:con:x} and \eqref{eq:mpc:con:u} which are in \cite{tebbe2025physics} modeled as datapoints $(t_i,z_i)$ with $z_i = \frac{1}{2}(z_\text{max} + z_\text{min})$ and heteroscedastic noise variance $\sigma_n^2(t_i) = \frac{1}{2}(z_\text{max} - z_\text{min})$.
This retains the property to compute the \ac{gp} posterior in closed form at the price of no guarantees to fulfill the constraints. \\
Moreover, the hypothesis space for the optimization \eqref{eq:mpc:obj} is determined by the \ac{se} kernel used in \cite{tebbe2025physics}. In the following, we provide approaches to improve the Control-as-Inference scheme by different latent kernel functions and additionally steering of the optimization function towards a quadratic cost function.
\section{Methodology}
\label{sec:method}
\subsection{Matérn covariance functions for LODE-GPs}
\label{sec:matern}
The choice of the latent covariance function of the \ac{lodegp} determines the degree of regularization we provide for the optimal control problem.
The benefit of using Matérn kernels is that the control function does not have to be smooth but can exhibit more control advantage behavior.
We provide a characterization of the function space we can model with \acp{gp} with Matérn kernels.
\begin{theorem} \label{thm:matern}
    The RKHS of a Matérn \ac{gp} $\mathcal{H}_{k_\nu}$ on $X \subset \mathbb{R}$
    , $\nu > m \geq 0$ with $\nu \in \mathbb{N}_0 + \tfrac12$ is dense in $C^m$.
\end{theorem}

\emph{Proof (Sketch):} The Matérn \ac{rkhs} $\mathcal{H}_{k_\nu}$ is equivalent to the Sobolev space $W^{s,2}$ and is, due to the Sobolev embedding theorem and standard density arguments, dense in $C^m$. \hfill $\blacksquare$ \\
Using the above property, we now show that \acp{lodegp} with a latent Matérn kernel are dense in the solution space of the underlying differential equation. For this purpose, we denote the relevant columns of $V$, i.e.\ the columns belonging to zero diagonal entries in $D$, by $P \in \mathbb{R}[\partial]^{q \times p}$ as a parametrization of the solution space.

\begin{theorem} \label{thm:lodegp}
For the parameterization $P \in \mathbb{R}[\partial]^{q \times p}$ as above, and Matérn kernel $k_{\nu}$ with $\nu \geq m + h + \tfrac12$ it holds that $P\cdot(\mathcal{H}_{k_{\nu}})^p$ is dense in $\text{sol}_{C^m}(H)$.
\end{theorem}
The proof for Theorem \ref{thm:lodegp} uses a similar argument as \cite[Proposition 3.5]{langehegermann2021linearly} and standard density arguments as those in Theorem \ref{thm:matern}.

Therefore, the mean of the \ac{lodegp} with latent Matérn kernel is able to approximate arbitrary solutions of given differentiability.
In particular, this result shows that all possible \emph{mean functions} are dense in the solution space.
This theorem allows us to use Matérn kernels in the \ac{lodegp} for our Control-as-Inference scheme.

\subsection{Cost Function optimization}
\label{sec:opt}

In \ac{mpc} we optimize arbitrary cost functions, but most of them are of quadratic type
\begin{subequations}
    \label{eq:mpc:cost}
\begin{align} 
    J(x,u) &= \sum_{i=1}^p (x - x_{\text{ref}})^TE(x - x_{\text{ref}}) \\
     &\qquad+ (u - u_{\text{ref}})^TF(u - u_{\text{ref}}) \\
           &= \sum_{i=1}^p (z - z_{\text{ref}})^T\begin{pmatrix}
        E & 0 \\
        0 & F
    \end{pmatrix}(z - z_{\text{ref}}) \\
    &= \sum_{i=1}^p (z - z_{\text{ref}})^TG(z - z_{\text{ref}}) = J(z) \label{eq:cost:G}
\end{align}
\end{subequations}
with $E,F,G$ being positive definite matrices of respective sizes.
The authors in \cite{tebbe2025physics} only provide optimality in the RKHS norm, given by the LODE-Kernel, which is hard to interpret. Therefore, we reweigh the \ac{lodegp} posterior $\mathcal{GP}(\mu^*, \Sigma^*)$ with the normal distribution $\mathcal{N}(z_\text{ref}, S)$ with $S$ being proportional to $G^{-1}$, therefore a positive definite matrix. The matrix $S$ is thereby acting in the same manner as the inverse of $G$ in \eqref{eq:mpc:cost}. We see this, as the log-likelihood of $\mathcal{N}(z_\text{ref}, S)$ is proportional to \eqref{eq:cost:G} with $S= G^{-1}$. \\
The reweighting is done via multiplication of the GP's predictive posterior $\mathcal{N}(\mu^*,\Sigma^*)$ with the multivariate normal distribution $\mathcal{N}(z_\text{ref}, S)$. The product is proportional to the normal distribution $\mathcal{N}(\mu_\text{opt}, \Sigma_\text{opt})$ with 
\begin{subequations}
\label{eq:cost}
\begin{align} 
    \mu_\text{opt} =& \Sigma_\text{opt} ({\Sigma^*}^{-1}\mu + S^{-1}z_\text{ref}) \label{eq:opt:mean} \\
    \Sigma_\text{opt} =& ({\Sigma^*}^{-1} + S^{-1})^{-1}, \label{eq:opt:cov}
\end{align}
\end{subequations}
due to \cite[Appendix A.2]{rasmussen2006gaussian}.
Thereby we obtain a normal distribution which still satisfies the linear ODE structure on the discretized grid. \\
This reweighting is possible for arbitrary cost functions, which are given via (log-)likelihoods, losing the property that the product of both distributions yields a Gaussian. \\ We may now sample from the normal distribution $\mathcal{N}(\mu_\text{opt}, \Sigma_\text{opt})$ for the sake of constraint satisfaction.

\subsection{Constraint satisfaction via sampling}
\label{sec:constraints}
A natural way of achieving constraint satisfaction, is combining the \ac{lodegp} posterior with the uniform likelihood 
\begin{equation} \label{eq:uniform}
    \mathcal{U}(z, z_{\min}, z_{\max}) = \begin{cases}
        \frac{1}{z_{\max} - z_{\min}}  &\text{ if $z \in [z_{\min}, z_{\max}]$} \\
        \centering 0  &\text { else}
    \end{cases}
\end{equation} on the constraints $[z_{\min}, z_{\max}]$. We obtain
\begin{equation} \label{eq:truncated}
    \mathcal{TGP} (z \vert t^*) = \frac{V\tilde{g}(t^*) \cdot \mathcal{U}(z, z_{\min}, z_{\max})}{\int_{z_{\min}}^{z_{\max}} V\tilde{g}(t^* \vert f) \mathrm{d}f}
\end{equation}
via Bayes' rule. This distribution is the well known truncated normal distribution \cite{bay2016generalization}.
Since calculating the denominator in \eqref{eq:truncated} is analytically intractable, we approximate this distribution via \ac{mc} sampling.
A naive way is i.i.d.\ sampling from the \ac{lodegp} posterior and rejecting samples which are not in the feasible interval given by the uniform distribution \eqref{eq:uniform}.
As i.i.d.\ sampling provides vanishing acceptance rates with growing dimension of the multivariate normal distribution, we require more efficient sampling strategies. \\
A sampling strategy which, by construction, provides acceptance rate $1$ is \ac{hmc} \cite{pakman2014exact}.
This sampling strategy leverages the sample distribution to a Hamiltonian dynamical system which is then simulated on paths which have a constant value of the Hamiltonian.
Exploiting the invariance structure of Hamiltonian systems and the normal distribution properties, the path of the Hamiltonian is simulated only inside the constraints resulting in proposed samples which satisfy the given linear constraints, given that the initial sample has support inside the constraints. 
Exact \ac{hmc} uses the fact that, for truncated Gaussians, the simulated Hamiltonian path can be integrated exactly with particle collisions at the constraint boundaries \cite{pakman2014exact}. The algorithm produces a Markov-Chain which is, up to discretization error, a sample of a \ac{lodegp} if the chain is initiated with a sample feasible to the constraints.
We choose this via the truncated mean
\begin{equation} \label{eq:mean:feasible}
    \bar{\mu}(t) = \begin{cases}
        z_{\min} \quad \text{if } \mu_\text{opt}(t) < z_{\min} \\
        z_{\max} \quad \text{if } \mu_\text{opt}(t) > z_{\max} \\
        \mu_\text{opt} \quad \text{ else}
    \end{cases}
\end{equation}
which does not necessarily satisfy the LODE structure. This choice is motivated by the result of \cite{bay2016generalization} which states that the optimal function based on its \ac{rkhs} for a truncated normal distribution is the posterior mode, which in our case is $\bar\mu$.
Thereby we provide an initial sample which is close to the \ac{lodegp} distribution and the optimal choice for a truncated normal distribution.
In similar contexts, \ac{hmc} is often referred to be one of the most efficient sampling methods for truncated multivariate normal distributions \cite{lopez2018finite, ding2022multivariate}.
We now show that, as the \ac{lodegp} is again a \ac{gp} which produces normal distributions, the truncation via uniform likelihoods keeps the linear \ac{ode} structure.
\begin{theorem}
    The support of the distribution $\mathcal{TGP}$ \eqref{eq:truncated} fulfills the ODE structure given by \eqref{eq:mpc:con:ode} and the constraints given by \eqref{eq:mpc:con:x} and \eqref{eq:mpc:con:u}.
\end{theorem}
\begin{proof}
    This theorem follows directly from Bayes' rule in \eqref{eq:truncated}. Since $V\tilde{g}$ has full support on solutions of the linear ODE, so does the posterior which is restricted on the constraints by the uniform likelihood.
\end{proof}
We choose \ac{hmc} as a sampler on a discretization of $\mathcal{TGP}(z\vert t)$, because it has an acceptance rate of $1$.
We provide our proposed methodology in Algorithm \ref{alg:GP_MPC}.

\begin{algorithm}[ht]
    \caption{\ac{lodegp} based \ac{mpc} with quadratic cost function optimization and constraint satisfaction.}%
    \KwIn{Initial state $x_{t_0}$ and control $u_{t_0}$, reference $x_\text{ref}$, constraints $z_{\min}, z_{\max}$, 
    weight matrix $S$, number of \ac{hmc} samples $n_{\text{samples}}$, discretization grid $t^*$}%
    \KwOut{Simulation/Control path $\{(x_{t}, u_{t}) | t \in [t_0, t_N]\}$ satisfying constraints \eqref{eq:mpc:con:ode} -- \eqref{eq:mpc:con:u}}
    Initialize \ac{lodegp} $V\tilde{g} \sim \mathcal{GP}(0,k_\nu)$ \\
    \For {$t_i=t_0, \dots, t_N$}{
        measure current state $x_{t_i}$ \\
        generate posterior $\mathcal{N}(\mu^*, \Sigma^*) = V\tilde{g}(t^* | \mathcal{D})$ \\
         Reweigh $\mathcal{N}(\mu^*, \Sigma^*)$ with $\mathcal{N}(z_\text{ref},S)$ via \eqref{eq:cost} to $\mathcal{N}(\mu_\text{opt}, \Sigma_\text{opt})$ \\
         Truncate $\mu_\text{opt}$ to constraints via \eqref{eq:mean:feasible} to $\mu_0$ \\
        \For {$j=1,\dots,n_{\text{samples}}$}
            {Use HMC to sample from $\mu_j$ from $\mu_{j-1}$}
        set control input $u_{t_i}$ as mean of samples
        }  
    \KwRet{$\{x_{t}, u_{t} | t \in [t_0, t_N]\}$}
    \label{alg:GP_MPC}
\end{algorithm}
\section{Experiments}
\label{sec:experiments}

\begin{figure*}[t]
    \centering
     \scalebox{0.75}{\input{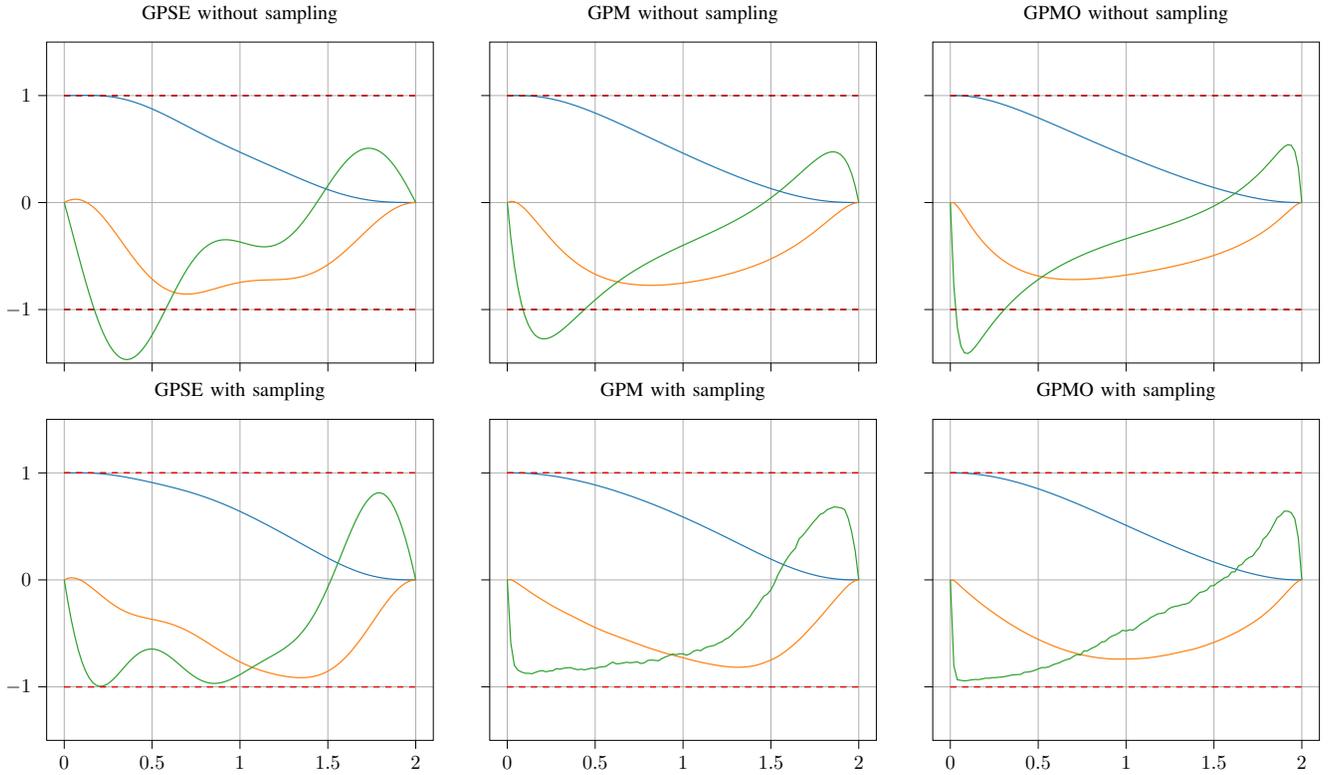}}
    \caption{Comparison of the outputs for the example given in Section \ref{sec:exp:spring}. The state $x_1$ is given in \textcolor{steelblue31119180}{blue}, the state $x_2$ in \textcolor{darkorange25512714}{orange} and the control $u$ in \textcolor{forestgreen4416044}{green}. In the first row, we see GPSE (left), GPM (middle), GPMO (right), in the second row we see the samples generated via HMC for the respective constrained LODE-GP.
    The constraints are given via the dashed lines at $z_{\min}=-1$ and $z_{\max} = 1$.
    We see that none of the three approaches satisfies the constraints without sampling.
    While GPSE provides smooth paths with small oscillations, the GPM and GPMO paths are rougher and have no oscillations.
    For the sampled versions, GPSE still suffers from oscillations via its smoothness requirements.
    The GPM and GPMO paths are rougher and obtain more control advantage behavior.}
    \label{fig:example1}
\end{figure*}

In this section we present the results of our proposed methodology on two systems.
We investigate the mean constraint violation
\begin{equation}\label{eq:constr_viol}
    \frac{1}{T}\sum_{i=1}^T  \max\{z(t_i) - z_{\max}, 0\} + \max\{z_{\min} - z(t_i), 0\}
\end{equation}
in order to show whether our approach can handle the imposed constraints.
Moreover, we investigate the mean control error, defined as
\begin{equation} \label{eq:control_error}
    \frac{1}{T}\sum_{i=1}^T (x(t_i) - x_{\text{ref}}(t_i))^2
\end{equation}
to demonstrate the control performance of our approach. \\
We compare a \ac{lodegp} model equipped with an \ac{se} kernel (GPSE) with a \ac{lodegp} model, equipped with a Matérn kernel, (GPM) and a \ac{lodegp} model, equipped with a Matérn kernel with additional cost optimization, (GPMO) as in Algorithm \ref{alg:GP_MPC}. 
For each model we provide a comparison of using the additional \ac{hmc} sampling step or not.

\subsection{Spring-mass example} \label{sec:exp:spring}
We start with a control task similar to the experiment in \cite{tebbe2025physics}.
In the first experiment, we use a spring mass damper system with the following unstable system, 
\begin{align}
	\dot{\begin{pmatrix}
	    x_1 \\
        x_2
	\end{pmatrix}} = \begin{pmatrix}
	0 & 1 \\
	1 & -1 
	\end{pmatrix} \begin{pmatrix}
	    x_1 \\
        x_2
	\end{pmatrix} + \begin{pmatrix}
	0 \\ \frac{5}{2}
	\end{pmatrix}
	u
\end{align}
with two integrators of the control function.
The control task is the regulation of all channels to $0$.
Therefore, we choose $z_\text{ref} = 0$ and $S$ as a diagonal matrix with $10^{-3}$ on the diagonal. The constraints are given by $z_{\min}=-1$ and $z_{\max}=1$.
As the authors in \cite{tebbe2025physics} used an SE kernel, the generated functions are smooth.
With the use of Matérn kernels, it is possible to control the smoothness of each individual kernel.
In our example of a spring-mass damper system, we have to use a kernel with differentiability $r \ge 2$, which models the state $x_1$. The function for $x_1$ is therefore at least twice differentiable, while the state $x_2$ is at least one time continuously differentiable and the control is only continuous. 
For all \acp{lodegp} we use hyperparameter optimization until convergence on the given data, in advance. We use a discretization of $101$ datapoints resulting in an equidistant grid with $\Delta t=0.02$. For the \ac{hmc} sampler, we use one chain with $200$ samples without a burn-in phase. 
The sampling, using a Python implementation of \cite{pakman2014exact}, took less than $1$ second for each model. \\
We see the results in Figure \ref{fig:example1} and Table \ref{tab:experiment1}.
The use of Matérn kernel in comparison to the \ac{se} kernel provides better performance, while satisfying the constraints with the use of the sampling step.
The additional reweighting via the quadratic cost function provides further performance gain while still adhering to the constraints after the sampling step. 

\begin{table}[h]
    \centering
     \caption{Results for Experiment 1.
     }
    \begin{tabular}{|c|c|c|c|}
        \hline
         Models without sampling & GPSE & GPM & GPMO  \\
         \hline
    Constraint error \eqref{eq:constr_viol} & 0.0203 & \textbf{0.01} & {0.0109} \\
    \hline
    Control error \eqref{eq:control_error}& 0.4024 & 0.3642 & \textbf{0.3367} \\
    \hline
        \hline
         Models with sampling & GPSES & GPMS & GPMOS  \\
         \hline
    Constraint error \eqref{eq:constr_viol} & \textbf{0.0} & \textbf{0.0} & \textbf{0.0} \\
    \hline
    Control error \eqref{eq:control_error}& 0.4294 & 0.3857 & \textbf{0.3452} \\
    \hline
    \end{tabular}
   
    \label{tab:experiment1}
\end{table}

\subsection{2D Integrator} \label{sec:exp:int}

\begin{figure*}
    \centering
    \scalebox{0.75}{\input{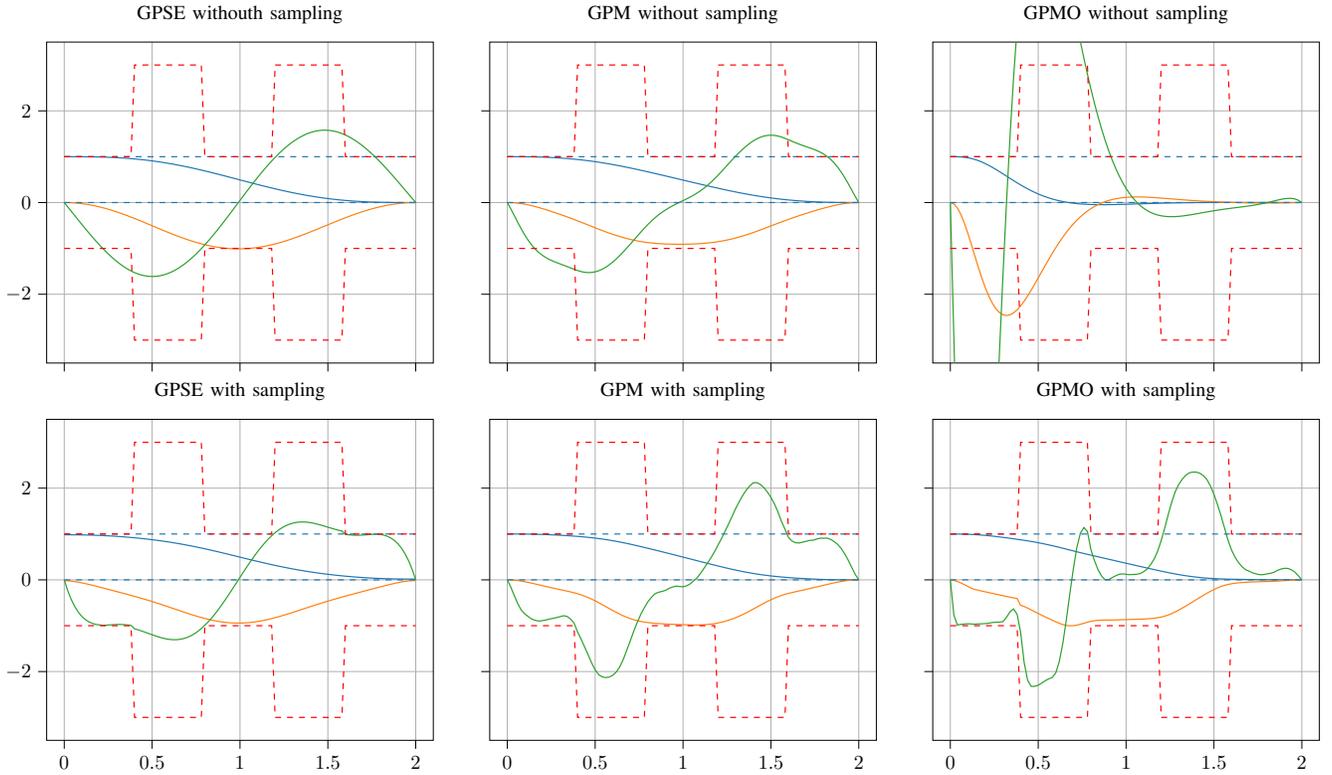}}
    \caption{Comparison of the outputs for the integrator example in Section \ref{sec:exp:int}. The state $x_1$ is given in \textcolor{steelblue31119180}{blue}, the state $x_2$ in \textcolor{darkorange25512714}{orange} and the control $u$ in \textcolor{forestgreen4416044}{green}. The constraint of $x_1$ is given via the dashed line in \textcolor{steelblue31119180}{blue} by $x_{\max} = 1$ and $x_{\min}=0$. The constraints of $x_2$ and $u$ are given via the dashed line in \textcolor{red}{red}. We observe a significantly better adapation to the constraints of GPM in comparison to GPSE. The smoothness of the \ac{se} kernel prevents the control function from aligning the constraint boundaries more aggressively. All models using sampling satisfy the constraints.
    }
    \label{fig:example2}
\end{figure*}
We use a simple integrator with varying constraints in order to check the adaptability of our control approach.
We use the system derived via
\begin{align}
	\dot{\begin{pmatrix}
	    x_1 \\
        x_2
	\end{pmatrix}} = \begin{pmatrix}
	0 & 1 \\
	0 & 0 
	\end{pmatrix} \begin{pmatrix}
	    x_1 \\
        x_2
	\end{pmatrix} + \begin{pmatrix}
	0 \\ 1
	\end{pmatrix}
	u
\end{align}
with more complex bounds which are shown in Figure \ref{fig:example2}. The control objective is to steer the state $x_1$, which is pictured in blue in Figure \ref{fig:example2}, from its initial value of $1$ to $0$ by time $t=2$. We choose $S$ as a diagonal matrix with $10^{-3}$ for state $x_1$ and $1$ for the rest of the diagonal entries. The remaining experimental settings are equal to the previous experiment. While the dynamics are simple, the constraints have a more complex structure than in the previous example. \\
We see in Figure \ref{fig:example2}, that the problem is hard to solve for the \ac{se} kernel based \ac{lodegp} GPSE, since this problem requires a lot of steering of the control function. This is better solved by a less differentiable function generated by the Matérn kernel based \ac{lodegp}.
While the \ac{hmc} sampling for the Matérn kernels takes, again, less than a second, the \ac{hmc} sampling for the \ac{se} kernel takes several minutes since the problem is nearly impossible to solve while satisfying the \ac{se} kernel's smoothness requirements.
The adaptation via a quadratic objective function yields a solution which violates the constraints, while the sampling procedure restores a feasible solution which performs better than both GPSE and GPM with sampling.
Again, the optimized approach provides the best control behavior at the cost of more constraint violation.
The sampling step provides simulation paths without constraint violation while losing only marginal control performance, see Table \ref{tab:experiment2} for results. \\
Both examples validate, that our method consisting of three unique steps performs well in the optimal control problems.

\begin{table}[ht]
    \centering
     \caption{Results for Experiment 2.
     }
    \begin{tabular}{|c|c|c|c|}
        \hline
         Models without sampling & GPSE & GPM & GPMO  \\
         \hline
    Constraint error \eqref{eq:constr_viol} & 0.0164 & 0.0190 & 0.5940 \\
    \hline
    Control error \eqref{eq:control_error}& 0.3837 & 0.3746 & \textbf{0.1314} \\
    \hline
        \hline
         Models with sampling & GPSES & GPMS & GPMOS  \\
         \hline
    Constraint error \eqref{eq:constr_viol} & \textbf{0.0} & \textbf{0.0} & \textbf{0.0} \\
    \hline
    Control error \eqref{eq:control_error}& 0.3677 & 0.3840 & \textbf{0.3373} \\
    \hline
    \end{tabular}
   
    \label{tab:experiment2}
\end{table}
\section{Conclusion}
\label{sec:conclusion}
We provided a Control-as-Inference method with constraint satisfaction guarantees for feasible \ac{mpc} problems. 
We do this by restricting the inference model, a linear ODE based Gaussian Process to input and state constraints via truncating the predictive multivariate Normal distribution. We prove that samples of the resulting distribution still satisfy the ODE. Additionally, we have shown that the Matérn kernel is a reasonable choice for the Control-as-Inference scheme, both from a theoretic view by proving a density theorem and from a practitioner's view in our experiments. Moreover, we provided a method to incorporate quadratic cost functions into the control scheme.
\bibliographystyle{plain}
\bibliography{literature}
\end{document}